\documentclass[11pt]{amsart}

\usepackage{amssymb}

\usepackage{amssymb,latexsym,amsmath,amsfonts}
\usepackage{eucal}
\usepackage{color}
\usepackage{ mathrsfs}
\usepackage{dsfont}
\usepackage{multirow}
\DeclareSymbolFont{SY}{U}{psy}{m}{n}
\DeclareMathSymbol{\emptyset}{\mathord}{SY}{'306}

\theoremstyle{plain}

\newtheorem{thm}{Theorem}[section]

\newtheorem{lem}[thm]{Lemma}
\newtheorem{prop}[thm]{Proposition}
\theoremstyle{definition}
\newtheorem{defn}[thm]{Definition}

\newtheorem{ex}[thm]{Example}

\numberwithin{equation}{section}
\def\A{{\mathbf A}}

\def\V{\mathbf V}

\def\beq{\begin{eqnarray}}
\def\eeq{\end{eqnarray}}
\def\beqa{\begin{eqnarray*}}
\def\eeqa{\end{eqnarray*}}

\begin{document}
\sloppy
\title{Contractivity and complete contractivity for finite dimensional Banach Spaces}

\author{Gadadhar Misra, Avijit Pal and Cherian Varughese}

\address[G. Misra]{Department of mathematics, Indian Institute of Science, Bangalore - 560 012, India}\email[G. Misra]{gm@math.iisc.ernet.in} 
\address[A. Pal]{Department of Mathematics and Statistics, Indian Institute of Science Education And Research Kolkata, Mohanpur - 741 246} \email[A. Pal]{avijitmath@gmail.com}
\address[C. Varughese]{Renaissance Communications, Bangalore - 560 058} \email[C. Varughese]{cherian@rcpl.com}
\thanks{\tiny The work of G.Misra was supported, in part, through the J C Bose National Fellowship and UGC-CAS.
The work of A. Pal was supported, in part, through the UGC-NET and
the IFCAM Research Fellowship. The results of this paper are taken from his PhD thesis, after significant simplifications, submitted to the Indian Institute of Science in  2014.}

\begin{abstract}
Choose an arbitrary but fixed set of $n\times n$ matrices $A_1, \ldots, A_m$ and let $\Omega_\mathbf A\subset \mathbb C^m$ be the unit ball with respect to the norm $\|\cdot\|_{\mathbf A},$ where $\|(z_1,\ldots ,z_m)\|_{\mathbf A}=\|z_1A_1+\cdots+z_mA_m\|_{\rm op}.$ 
It is known that if 
$m\geq 3$ and $\mathbb B$ is any ball in $\mathbb C^m$ with
respect to some norm, say $\|\cdot\|_{\mathbb B},$ then there
exists a contractive linear map $L:(\mathbb
C^m,\|\cdot\|^*_{\mathbb B})\to \mathcal M_k$ which is
not completely contractive. The characterization of those balls in
$\mathbb C^2$ for which contractive linear maps are always
completely contractive thus remains open. We answer this question for
balls of the form $\Omega_\mathbf A$ in $\mathbb C^2.$
\end{abstract}
\maketitle
\section{Introduction}

In 1936 von Neumann (see \cite[Corollary 1.2]{paulsen})
proved that if $T$ is a bounded linear operator on a separable
complex Hilbert space $\mathcal H,$ then, for all complex polynomials $p$,
$$\|p(T)\|\leq \|p\|_{\infty, \mathbb D}:=\sup\{|p(z)|: |z|<1\}$$ if and only if $\|T\|\leq 1.$
Or, equivalently,  the homomorphism $\rho_{T}$
induced by $T$ on the polynomial ring $ P[z]$ by the rule
$\rho_{T}(p)=p(T)$ is contractive if and only if $T$ is contractive.

The original proof of this inequality is intricate. A couple of
decades later, Sz.-Nazy (see \cite[Theorem 4.3]{paulsen}) proved that a bounded linear operator $T$ admits a
unitary (power) dilation if and only if there exists a unitary
operator $U$ on a Hilbert space $\mathcal K \supseteq \mathcal H$
such that
$$P_\mathcal H\,p(U)_{|\mathcal H}=p(T),$$ for all polynomials $p.$
The existence of such a dilation may be established by actually
constructing a unitary operator $U$ dilating $T.$ This
construction is due to Schaffer (cf. \cite{shaffer}). Clearly, the von
Neumann inequality follows from the existence of a power dilation
via the spectral theorem for unitary operators.

Let $P=\left(\!(p_{ij})\!\right)$ be a  $k\times k$ matrix valued polynomial in $m$ variables. Let
$$\|P\|_{\infty,\Omega}=\sup\{\|\left(\!(p_{ij}(z))\!\right)\|_{\rm op}: z \in
\Omega\},$$ where $\Omega \subseteq \mathbb C^m$ is a bounded open and connected set. 
Define $P(T)$ to be the operator
$\left(\!(p_{ij}(T))\!\right),$ $1\leq i,j \leq k.$ The homomorphism $\rho_T$  is said
to be completely contractive if 
$$\|P(T)\| \leq \|P\|_{\infty,\Omega},\,\, k=1,2, \ldots .$$

A deep theorem due to Arveson (cf. \cite{AW}) says that $T$ has a\emph{ normal boundary dilation} if and only if $\rho_T$ is completely contractive.  Clearly, if $\rho_T$ is completely contractive, then it is
contractive. The dilation theorems due to Sz.-Nazy and Ando (cf. \cite{paulsen}) give
the non-trivial converse in the case of the disc and the
bi-disc algebras. 

%

However, Parrott (cf. \cite{parrott}) showed that there are three commuting contractions for which it is impossible to find commuting unitaries dilating them. 
In view of Arveson's theorem this naturally leads to the question of finding other algebras $\mathcal O(\Omega)$ for which all contractive homomorphisms are necessarily completely contractive.
At the moment, this is known to be true of the disc, bi-disc (cf. \cite{paulsen}),  symmetrized bi-disc (cf. \cite{young}) and the annulus algebras (cf. \cite{agler}). Counter examples are known for domains of connectivity $\geq 2$ (cf. \cite{michael}) and any ball in $\mathbb C^m,$ $m\geq 3,$ as we will explain below.

Neither Ando's proof of the existence of a unitary dilation for a
pair of commuting contractions, nor the counter example to such an
existence theorem due to Parrott involved the notion of complete
contractivity directly. In the papers
\cite{G, GM, sastry}, it was shown that the examples of Parrott are not even $2$ -- contractive. In these papers, for any bounded, connected and open set $\Omega \subset \mathbb C^m,$  the homomorphism $\rho_\V: \mathcal O(\Omega) \to \mathcal M_{p+q},$ 
induced by  an  $m$-tuple of $p\times q$ matrices $\V=(V_1, \ldots ,V_m),$ modeled after the examples of Parrott, was introduced. 
This was further studied, in depth, by V. Paulsen \cite{vern}, where he
showed that the question of ``contractive vs completely  contractive''
for Parrott like homomorphisms $\rho_{\V}$ is equivalent to the question of ``contractive vs completely contractive'' for the linear maps $L_{\V}$ from some finite dimensional Banach space $X$ to $\mathcal M_n(\mathbb
C).$ The existence of linear maps of the form $L_{\mathbf V}$ which are contractive but not completely contractive  for $m\geq 5$ were found by him. A refinement  (see remark at the bottom of p. 76 in \cite{pisier})  includes the case  $m=3,4,$ leaving the question of what happens when $m=2$ open. This is Problem 1 on page 79 of \cite{pisier} in the list of ``Open Problems''. 

For the normed linear space $(\mathbb C^2,\|\cdot\|_\mathbf A),$ we show, except when the pair $A_1, A_2$ is simultaneously diagonalizable,  that there is a contractive linear map on $(\mathbb C^2,\|\cdot\|_\mathbf A)$ taking values in $p \times q$ matrices,  which is not completely contractive. 

We point out that the results of Paulsen
used deep ideas from geometry of finite dimensional Banach spaces.
In contrast, our results are elementary in nature, although the
computations, at times, are somewhat involved.
\section{Preliminaries}
The norm  $\|\mathbf z\|_\mathbf A =\| z_1 A_1 + \cdots +z_m A_m\|_{\rm op},\: \mathbf z\in \mathbb C^m,$ is obtained from the embedding of the linear space $\mathbb C^m$ into the $C^*$ algebra of $n\times n$ matrices via the map $P_\mathbf A(z):=  z_1 A_1 + \cdots + z_m A_m.$   
Let $\Omega_\mathbf A\subset \mathbb C^m$
be the unit ball with respect to the norm $\|\cdot\|_{\mathbf A}.$  Let $\mathcal O(\Omega_\mathbf A)$ denote the algebra of functions each of which is holomorphic on some open set containing the closed unit ball $\bar{\Omega}_\A$.
Given $p\times q$ matrices $V_1, \ldots, V_m$ and a function $f
\in \mathcal O(\Omega_\mathbf A)$, define 
\begin{equation}\label{8}
\rho_{\V}(f):=\left (
\begin{smallmatrix}
f(w)I_p~&~ \sum_{i=1}^{m} \partial_if(w)~V_{i} \\
0  ~&~ f(w)I_q
\end{smallmatrix}\right )~~~~{\rm for~a~fixed} ~w\in\Omega_\A.
\end{equation} 
 Clearly, $\rho_{\V}:(\mathcal O(\Omega_\A), \|\cdot\|_\infty) \rightarrow (\mathcal M_{p + q}(\mathbb C),\|\cdot\|_{\rm op})  $ defines an algebra homomorphism.


At the outset we point out the interesting and useful fact that $\rho_\V$ is contractive on $\mathcal O(\Omega_\A)$ if and only if it is contractive on the subset of functions which vanish at $w$.  This is the content of the following lemma. The proof is reproduced from \cite[Lemma 5.1]{vern}, a direct proof appears in \cite[Lemma 3.3]{G}. 

\begin{lem}\label{9}
$\sup_{\|f\|_\infty =1}\{\|\rho_\V(f)\|_{\rm op}: f\in \mathcal O(\Omega_\A) \}\leq 1$ if and only if $\sup_{\|g\|_\infty =1}\{\|\rho_\V(g)\|_{\rm op}: g\in \mathcal O(\Omega_\A) , g(w)=0\} \leq 1.$
\end{lem}

\begin{proof}
The implication in one direction is obvious.  
To prove the converse, assume that $\|\rho_\V(g)\|\leq 1$ for every $g$ such that $g(w)=0$ and $\|g\|_\infty=1$.

For $f\in \mathcal O (\Omega_\A)$ with $\|f\|_\infty = 1$ let $\phi_{f(w)}$ be the M\"{o}bius map of the disc which maps $f(w)$ to $0$.  We let $g=\phi_{f(w)}\circ f$.  Then $g(w)=0, \|g\|_\infty=1$ and, from our assumption, $\|\rho_\V(g)\|\leq 1$. So
\begin{align*}
\|\rho_\V(f)\| &= \|\rho_\V(\phi_{f(w)}^{-1}\circ g \|\\
               &= \|\phi_{f(w)}^{-1}\big(\rho_\V(g)\big) \|~~ {\rm since}~\rho_\V ~ {\rm is~ a ~homomorphism}   \\
               &\leq 1.           
\end{align*}
In the last step we use the von Neumann inequality since $\phi_{f(w)}^{-1}$ is a rational function from the disc to itself.
\end{proof}


\noindent{\bf Note:} For the rest of this work, we restrict to the case where $w=0$ in the definition (\ref{8}) of $\rho_\V$ above.

The following lemma provides a characterization of the unit ball $\Omega_\A^*$ with respect to the dual norm $\|\cdot \|_{\A}^*$ in $\mathbb C^m,$ that is $\Omega_{\A}^* = (\mathbb C^m,\|\cdot\|_{\A}^*)_1$.

\begin{lem}\label{10}
The dual unit ball $$\Omega_\A^* = \big\{\big(\partial_1 f(0), \partial_2 f(0), \cdots ,\partial_m f(0)\big): f\in{\rm Hol}(\Omega_\A,\mathbb D),f(0)=0\big\}.$$
\end{lem}
\begin{proof}
Given ${\mathbf z}\in \mathbb C^m$ such that $\|\mathbf z\|_\A = 1$ and $f\in{\rm Hol}(\Omega_\A,\mathbb D),f(0)=0$, we define $g_{\mathbf z}:\mathbb D \rightarrow \Omega_\A$ by
$$g_{\mathbf z}(\lambda) = \lambda {\mathbf z}, ~\lambda \in \mathbb D.$$
Then $f\circ g_{\mathbf z}: \mathbb D \rightarrow \mathbb D$ with $(f\circ g_{\mathbf z})(0)=0$.  
Applying the Schwarz Lemma to the function $(f\circ g_{\mathbf z})$ we get
$$1\geq |(f \circ g_\mathbf z)^{\prime}(0)|=|f^{\prime}(g_\mathbf z(0))\cdot g^{\prime}_\mathbf z(0)|=
|f^{\prime}(0)\cdot g_{\mathbf z}^{\prime}(0)|=|f^{\prime}(0)\cdot \mathbf z|.$$
In the above, $f^{\prime}(0)\cdot \mathbf z = \sum_{i=1}^{m} (\partial_if(0))z_{i},$ etc.

Hence $\big(\partial_1 f(0), \partial_2 f(0), \cdots ,\partial_m f(0)\big)\in \Omega_\A^*$.

Conversely, given $\mathbf w \in \Omega_\A^*$, we define $f_{\mathbf w}(\mathbf z) = \mathbf w \cdot \mathbf z$ so that $\partial_i f_{\mathbf w}(0) = w_i$.
\end{proof}

\subsection{The Maps $L_{\V}^{(k)}$:}

From Lemma \ref{9} above it follows that 
\begin{equation}\label{11}
\|\rho_\V\| \leq 1  ~{\rm if ~and ~only~ if}~  \sup_{\|f\|_\infty=1,f(0)=0}~~\|\sum_{i=1}^{m} \partial_if(0)~V_{i}\|_{\rm op}\leq 1.
\end{equation}
Considering Lemma \ref{10} and the equivalence (\ref{11}) above it is natural to consider the induced linear map $L_{\V}:(\mathbb
C^m,\|\cdot\|^*_{\mathbf A})\to \mathcal M_{p, q}(\mathbb
C)$ given by
$$L_{\V}(w)=w_1V_1+\cdots+w_mV_m.$$ 
It follows from (\ref{10}) above that 
$$\|\rho_\V\|\leq 1 ~{\rm if~ and~ only~ if}~ \|L_\V\| \leq 1.$$

We will show now that the complete contractivity of $\rho_\V$ and $L_\V$ are also related similarly.

For a holomorphic function $F:\Omega_\A \rightarrow \mathcal M_k$ with $\|F\|=\sup_{\bf z\in\Omega_{\A}}\|F(\bf z)\|,$ we define
\begin{equation}\label{13}
\rho_{\V}^{(k)}(F):=(\rho_{\V}(F_{ij}))_{i,j=1}^m=\left ( \begin{smallmatrix}
F(0)\otimes I ~&~ \sum_{i=1}^m (\partial_i F(0))\otimes V_i \\
0  ~&~ F(0)\otimes I
\end{smallmatrix}\right ).
\end{equation}
Using a method similar to that used for $\rho_\V$ it can be shown that 
$$\|\rho_\V^{(k)}\| \leq 1 ~{\rm if~and~only~if}~ \sup_F\{\|\sum_{i=1}^m (\partial_i F(0))\otimes V_i\|: F\in{\rm Hol} (\Omega_{\A},(\mathcal M_k)_1), F(0) = 0\}\leq 1,$$
that is, (by repeating the argument used for $\rho_\V$) we have
$$
\|\rho_\V^{(k)}\| \leq 1 ~{\rm if~and~only~if}~ \|L_\V^{(k)}\| \leq 1,
$$
where $$L_\V^{(k)}: (\mathbb C^m \otimes \mathcal M_k,\|\cdot \|_{\A,k}^*) \rightarrow (\mathcal M_k \otimes \mathcal M_{p,q},\|\cdot \|_{\rm op})$$
is the map
$$
L_{\V}^{(k)}(\Theta_1,\Theta_2,\cdots, \Theta_m)=\Theta_1\otimes V_1 + \Theta_2\otimes V_2 + \cdots +\Theta_m\otimes V_m ~~~{\rm for}~ \Theta_1,\Theta_2,\cdots, \Theta_m \in \mathcal M_k.
$$
\subsection{The polynomial $P_\A$}
A very useful construct for our analysis is the matrix valued polynomial $P_\A$ with $P_\A(\Omega_\A) \subseteq (\mathcal M_n, \|\cdot\|_{\rm op})_1$ defined by 
$$ P_\A (z_1,z_2, \cdots, z_m) = z_1 A_1 + z_2A_2+\cdots+ z_mA_m,$$
with the norm 
$\|P_\A\|_{\infty}= \sup_{(z_1,\cdots,z_m)\in \Omega_\A}\|P_\A(z_1,\cdots, z_m)\|_{\rm op}$.  Note that $\|P_\A\|_{\infty} =1$ by definition.  The typical procedure used to show the existence of a homomorphism which is contractive but not completely contractive is to construct a contractive homomorphism $\rho_\V$ (by a suitable choice of $\V$) and to then show that its evaluation on $P_\A$, that is, $\rho_\V^{(n)}(P_\A)$, has norm greater than 1.

\subsection{Homomorphisms induced by $m$-vectors}
We now consider the special situation when the matrices $V_1,\cdots, V_m$ are vectors in $\mathbb C^m$ realized as row $m$-vectors.  For $w= (w_1, \ldots, w_m)$ in some bounded domain $\Omega\subseteq \mathbb C^m,$ the commuting $m$-tuple of $(m+1) \times (m+1)$  matrices of the form $\Big ( \begin{smallmatrix} w	_i & V_i \\ 0 & w_i I_m \end{smallmatrix} \Big ),$ $1 \leq i \leq m,$ induce the homomorphism $\rho_{\boldsymbol V}$ via the usual functional calculus, that is, 
$$\rho_{\boldsymbol V}(f) := f\big (\Big ( \begin{smallmatrix} w	_1 & V_1 \\ 0 & w_1 I_m \end{smallmatrix}\Big ),\ldots , \Big ( \begin{smallmatrix} w	_m & V_m \\ 0 & w_m I_m \end{smallmatrix} \Big ) \big ), f \in \mathcal O(\Omega),$$ 
see \eqref{8}.  The  localization of a commuting $m$ - tuple $\boldsymbol T$ of  operators  in the class $B_1(\Omega),$ introduced in  (\cite{cowen, douglas}), is also a commuting $m$ - tuple of $(m+1)\times (m+1)$ matrices, which is exactly of the form described above. The  vectors $V_1, \ldots , V_m$ appearing in such localizations  are  given explicitly in terms of the curvature of the holomorphic Hermitian vector bundle  corresponding to $\boldsymbol T$ as shown in \cite{douglas}.  The contractivity of the homomorphism $\rho_{\boldsymbol V}$ then results in curvature inequalities (see \cite{misra,GM,sastry,avijit}).    

Let $V_i = \left(\begin{matrix} v_{i1} & v_{i2} &\cdots &v_{im} \end{matrix}\right),~~i=1,\cdots, m$. The propositions below are useful to study contractivity and complete contractivity in this special case, where, as before, we assume that $\Omega=\Omega_{\boldsymbol A}$  and $w=0.$ 

\begin{prop}\label{14}
The following are equivalent:

\begin{enumerate}

 \item[(i)] $\rho_{\V}$ is
contractive,

\item [(ii)] $\sup_{\sum_{j=1}^{m}|z_j|^2 \leq 1}
\|\sum_{j=1}^{m}z_jB_{j} \|_{\rm op}^2 \leq 1,$ where
$B_{j}=\sum_{i=1}^{m}v_{ij}A_{i}.$
\end{enumerate}
\end{prop}

\begin{proof}
We have shown that the homomorphisms  $\|\,\rho_{\V}\,\|_{\mathcal O(\Omega_{\mathbf A})\rightarrow\mathcal M_{m+1}(\mathbb C)}$  is contractive if and only if the linear map $\|L_{\V}\|_{(\mathbb C^m,
\|\,\cdot\,\|_{\A}^*) \rightarrow (\mathbb C^m,
\|\,\cdot\,\|_2) }$ is contractive (equivalently if $\|L_{\V}^*\|_{(\mathbb C^m,\|\,\cdot\,\|_2)  \rightarrow (\mathbb C^m,\|\,\cdot\,\|_\A) }$ is contractive).

The matrix representation of
$L_{\V}^*$ is  $\left (
\begin{smallmatrix}
v_{11} & \ldots & v_{1m}  \\
\vdots & \ddots &\vdots\\
 v_{m1} &\ldots   & v_{mm}
\end{smallmatrix} \right ).$ 

Hence the contractivity of $L_{\V}^*$ is given by the condition that
$$\sup_{\sum_{j=1}^m |z_j|^2 \leq 1}\Bigg\|\left (
\begin{smallmatrix}
v_{11} & \ldots & v_{1m}  \\
\vdots & \ddots &\vdots\\
 v_{m1} &\ldots   & v_{mm}
\end{smallmatrix} \right )\left ( \begin{smallmatrix}
z_1 \\
 \vdots\\
 z_m
\end{smallmatrix} \right )\Bigg\|_\A \leq 1 .$$

From the definition of $\|\cdot\|_\A$ it follows that 

$$\|L_{\V}^*\|_{(\mathbb C^m, \|\,\cdot\,\|_2)\rightarrow
(\mathbb C^m, \|\,\cdot\,\|_{\A})} \leq 1 ~~{\rm if~and~only~if}~~ \sup_{\sum_{j=1}^{m}|z_j|^2 \leq 1}
\|\sum_{j=1}^{m}z_j B_{j} \|_{\rm op}^2 \leq 1$$ where
$B_{j}=\sum_{i=1}^{m}v_{ij}A_{i}.$ 
\end{proof}
In particular, if $V_1 = \left(\begin{matrix} u & 0 \end{matrix}\right)$ and $V_2 = \left(\begin{matrix} 0 & v \end{matrix}\right),$ the condition (ii) above becomes $$\sup_{\sum_{j=1}^{2}|z_j|^2 \leq 1}
\|z_1 u A_1 + z_2 v A_2 \|^2 \leq 1,$$ which is equivalent to the following two conditions:

\begin{enumerate}
\item[(i)] $|u|^2\leq \frac{1}{\|A_1^{*}\|^2}$ or $|v|^2\leq \frac{1}{\|A_2^{*}\|^2}$
\item[(ii)] 
\begin{equation*}
\inf_{\beta\in\mathbb C^2,\|\beta\|=1}\Big\{1-|u|^2\|A_{1}^*\beta\|^2
-|v|^2\|A_{2}^*\beta\|^2+
|uv|^2\Big(\|A_{1}^*\beta\|^2\|A_{2}^*\beta\|^2 - |\left\langle
A_1A_{2}^*\beta, \beta\right\rangle |^2\Big)\Big\}\geq 0.
\end{equation*}

\end{enumerate}

\begin{prop}\label{15}
The following are equivalent:
\begin{enumerate}
\item[(i)]$\|\rho_{\V}^{(n)}(P_{\mathbf A})\|\leq 1,$

\item[(ii)] $The~~ n\times mn ~~matrix \left( \begin{matrix} B_1 & B_2 & \cdots &B_m \end{matrix}\right) $ is contractive,  where $B_j=\sum_{i=1}^{m}v_{ij}A_i$.

\end{enumerate}
\end{prop}

\begin{proof}
Since
$P_{\mathbf A}(0)=0,$ it follows from the definition (\ref{13})
that $\|\rho_{\V}^{(n)}(P_{\mathbf A})\|\leq 1$ if and only if
$$\|A_1\otimes V_1+\ldots+A_m\otimes V_m\|\leq 1.$$
For $V_i=\left(\begin{matrix}v_{i1}&\cdots& v_{im}\end{matrix}\right),$ we have
$$A_1\otimes V_1+\ldots+A_m\otimes V_m = \left( \begin{matrix} B_1 & B_2 & \cdots &B_m \end{matrix}\right)$$

Thus $\|\rho_{V}^{(n)}(P_{\mathbf A})\|\leq 1$ if and only if
$\| \left( \begin{matrix} B_1 & B_2 & \cdots &B_m \end{matrix}\right)
 \|\leq 1.$

\end{proof}

In particular if $V_1 = \left(\begin{matrix} u & 0 \end{matrix}\right)$ and $V_2 = \left(\begin{matrix} 0 & v \end{matrix}\right)$ the condition (ii) above becomes
$$\inf_{\beta\in\mathbb C^2,\|\beta\|=1}\Big\{1-|u|^2\|A_{1}^*\beta\|^2
-|v|^2\|A_{2}^*\beta\|^2\Big\}\geq 0.$$


\noindent{\bf Note:} For most of this paper we will restrict to the two dimensional case. That is, we consider $\mathbb C^2$ with the norm defined by a matrix pair $(A_1,A_2)$. In fact, for the most part, we even restrict to the situation where $A_1, A_2$ are $2\times 2$ matrices. This is adequate for our primary purpose of constructing homomorphisms of $\mathcal O(\Omega_\A)$ which are contractive but not completely contractive.  Many of the results can be adapted to higher dimensional situation. 

\section{Defining Function and Test Functions}

Recall the matrix valued polynomial $P_{\mathbf A}:\Omega_{\mathbf A}\rightarrow (\mathcal M_{2}, \|\cdot\|_{\rm op})_{1}$  defined earlier by $$P_{\mathbf A}(z_{1}, z_{2})=z_{1}A_{1}+z_{2}A_{2},$$ where $(\mathcal M_{2}, \|\cdot\|_{\rm op})_{1}$ is the matrix unit ball with respect to the operator norm. For $(z_{1}, z_{2})$ in
$\Omega_{\mathbf A},$ the norm $$\|P_{\mathbf A}\|_\infty :=
\sup_{(z_1,z_2)\in \Omega_\mathbf A} \|P_\mathbf A(z_1,z_2)\|_{\rm op} = 1$$ by definition of the
polynomial $P_\mathbf A.$ 

Let $\mathbb B^2$ be the unit ball in $\mathbb C^2$. For $(\alpha, \beta) \in \mathbb B^2 \times \mathbb B^2,$  define
$p^{(\alpha,\beta)}_{\mathbf A}:\Omega_{\mathbf A}\rightarrow
\mathbb D$ to be the linear map $$p^{(\alpha,\beta)}_{\mathbf
A}(z_1,z_2)=\left\langle P_{\mathbf A}(z_1, z_2)\alpha, \beta\right\rangle
=z_1\langle A_1\alpha,\beta \rangle + z_2\langle A_2 \alpha,\beta \rangle.$$
The sup norm $\|p^{(\alpha,\beta)}_{\mathbf A}\|_{\infty}$, for any pair of vectors $(\alpha, \beta)$ in
$\mathbb B^2\times \mathbb B^2,$ is at most $1$ by definition. 
Let $\mathcal{P}_{\mathbf A}$ denote the collection of linear functions
$\{p^{(\alpha,\beta)}_{\mathbf A}:(\alpha, \beta)\in \mathbb B^2\times \mathbb B^2\}.$

The map $P_\mathbf A$ which we call the `Defining Function' of the domain and the collection of functions $\mathbf{\mathcal{P}}_\mathbf A$ which we call a family of `Test Functions' encode a significant amount of information relevant to our purpose about the homomorphism $\rho_\V$.  For instance $\rho_\V$ is contractive if its restriction to $\mathcal{P}_\mathbf{A}$ is contractive. Also the lack of complete contractivity can often be shown by evaluating $\rho_\V^{(2)}$ on $P_\mathbf{A}$.  Some of the details are outlined in the lemma below.

\begin{lem}\label{lemma1} In the notation fixed in the preceding discussion, we have  
\begin{align*}
&(i)~~\sup_{\|\alpha\|=\|\beta\|=1}\|\rho_\V(p_\A^{(\alpha,\beta)})\| \leq \|\rho_\V^{(2)}(P_\A)\|,\\ 
&(ii)~~\rho_\V~~\mbox{\it  is contractive if and only if~~} \sup_{\|\alpha\|=\|\beta\|=1}\|\rho_\V(p^{(\alpha,\beta)}_\mathbf A)\| \leq 1.
\end{align*}

\end{lem}

\begin{proof}[Proof of (i)]
Since 
$$\rho_\V(p_\A^{(\alpha,\beta)})=
\left(\begin{array}{cc}
0 & (\partial_1 p_\A^{(\alpha,\beta)}(0))~ V_1 + (\partial_2 p_\A^{(\alpha,\beta)}(0))~ V_2 \\
0 & 0 \\
\end{array}\right)$$
by definition, it follows that
\begin{eqnarray}
\nonumber \|\rho_\V(p_\A^{(\alpha,\beta)})\| &=&\|(\partial_1 p_\A^{(\alpha,\beta)}(0))~ V_1 + (\partial_2 p_\A^{(\alpha,\beta)}(0))~ V_2\|_{\rm op}\\ \nonumber
&=&\|\langle A_1\alpha,\beta\rangle ~V_1 + \langle A_2\alpha,\beta\rangle~ V_2\|_{\rm op}\\ \nonumber
&=&\sup_{\|u\|=\|v\|=1}~~
| \langle A_1 \alpha,\beta\rangle \langle V_1u, v\rangle +  \langle A_2 \alpha,\beta\rangle \langle V_2u,v\rangle |. 
\end{eqnarray}
Hence
\begin{eqnarray}
\sup_{\|\alpha\|=\|\beta\|=1}\|\rho_\V(p^{(\alpha,\beta)}_\mathbf A)\| \label{1}   &=& \sup_{\|\alpha\|=\|\beta\|=1}~~\sup_{\|u\|=\|v\|=1}~~
| \langle A_1 \alpha,\beta\rangle \langle V_1u, v\rangle +  \langle A_2 \alpha,\beta\rangle \langle V_2u,v\rangle |  \\ \nonumber
&=&\sup_{\|\alpha\|=\|\beta\|=1}~~\sup_{\|u\|=\|v\|=1}~~
| \langle (A_1 \otimes V_1 + A_2 \otimes V_2)\alpha \otimes u, \beta \otimes v\rangle | \\ \nonumber
&=&\sup_{\|\alpha\|=\|\beta\|=1}~~\sup_{\|u\|=\|v\|=1}~~
| \langle \rho_\V^{(2)}(P_\A)\alpha \otimes u, \beta \otimes v\rangle | \\ \nonumber
&\leq&\|\rho_\V^{(2)}(P_\A)\|.
\end{eqnarray}
\end{proof}
\begin{proof}[Proof of (ii)]
As indicated earlier the contractivity of $\rho_\V$ is equivalent to the contractivity of 
$$L_\V:(\mathbb{C}^2,\|\cdot \|_{\A}^*)\rightarrow (\mathcal{M}_{p,q},\|\cdot\|_{\rm op})$$
given by the formula 
$$L_\V(\omega_1,\omega_2)=\omega_1 V_1 + \omega_2 V_2.$$
 So we identify the conditions for the contractivity of $L_\V:$
{\small\begin{align}\|L_\V\|\nonumber&=\sup_{\|(\omega_1,\omega_2)\|_{\mathbf{A}}^*\leq 1}~~\|\omega_1 V_1+\omega_2 V_2\|_{\rm op} \\ \nonumber
&=\sup_{\|(\omega_1,\omega_2)\|_{\mathbf{A}}^*\leq 1}~~~\sup_{\|u\|=\|v\|=1}~~
|\omega_1 \langle V_1u, v\rangle+ \omega_2 \langle V_2u,
v\rangle|.\nonumber
\end{align}}
Hence, since $(\omega_1,\omega_2)$ lies in the dual of $\Omega_\A$,
\begin{align}
\|L_\V\| \leq 1 \nonumber &\iff (\langle V_1 u,v\rangle , \langle V_2 u,v\rangle) \in \Omega_\A ~~~~\forall u,v ~~{\rm such~ that}~~ \|u\|=\|v\|=1\\ \nonumber
&\iff \sup_{\|u\|=\|v\|=1}~~
\| \langle V_1u, v\rangle A_1 + \langle V_2u,
v\rangle A_2 \|_{\rm op} \leq 1 \\ \nonumber
&\iff \sup_{\|\alpha\|=\|\beta\|=1}~~\sup_{\|u\|=\|v\|=1}~~
| \langle A_1 \alpha,\beta\rangle \langle V_1u, v\rangle +  \langle A_2 \alpha,\beta\rangle \langle V_2u,v\rangle |\leq 1\\ \nonumber
&\iff \sup_{\|\alpha\|=\|\beta\|=1}\|\rho_\V(p^{(\alpha,\beta)}_\mathbf A)\| \leq 1~~~ {\rm from}~ (\ref{1})~{\rm above.} 
\end{align}
\end{proof}
As mentioned earlier, by choosing a pair $(V_1,V_2)$ such that the inequality in $(i)$ above is strict, we can often construct a contractive homomorphism which is not completely contractive. We illustrate below choices of $(V_1,V_2)$ for the Euclidean ball for which the inequality is strict.

\begin{ex}(Euclidean Ball)
Choosing $\mathbf A=\left(\left (
\begin{smallmatrix}
1 & 0   \\
0  &  0
\end{smallmatrix}\right ) ,
\left ( \begin{smallmatrix}
0&  1\\
0  &  0
\end{smallmatrix}\right )\right),$ we see that  $\Omega_\mathbf A$
defines the Euclidean ball $\mathbb{B}^2$ in $\mathbb C^2$. Choose
$V_1=(v_{11}\,\, v_{12}), V_2= (v_{21}\,\, v_{22}).$ We will prove
that
$$\sup_{\|\alpha\|=\|\beta\|=1}\| \rho_{\V}(p^{(\alpha,
\beta)}_\mathbf A) \|< \|\rho_{\V}^{(2)}(P_{\mathbf A})\|_{\rm op},$$
if $V_1$ and $V_2$ are linearly independent. 

In fact we can choose $(V_1,V_2)$ such that $\sup_{\|\alpha\|=\|\beta\|=1}\| \rho_{\V}(p^{(\alpha,
\beta)}_\mathbf A) \|\leq 1$ and $\|\rho_{\V}^{(2)}(P_{\mathbf A})\|_{\rm op}>1$.
This example of a contractive homomorphism of the ball algebra which
is not completely contractive was found in \cite{G, GM}.
\begin{thm}\label{thm1} For $\Omega_\mathbf A=\mathbb B^2,$  let $V_1 = \left(\begin{matrix}v_{11} & v_{12}\end{matrix}\right), V_2 = \left(\begin{matrix}v_{21} & v_{22}\end{matrix}\right)$. Then 
\begin{align*}
&(i)~\sup_{\|\alpha\|=\|\beta\|=1}\| \rho_{\V}(p^{(\alpha
,\beta)}_\mathbf A) \|^2 = \|\left(\begin{smallmatrix}v_{11}~&v_{12}\\
v_{21}~&v_{22}\end{smallmatrix}\right)\|_{\rm op}^2\\
 &(ii)~ \|\rho_{\V}^{(2)}(P_{\mathbf A})\|^2_{\rm op} = \|\left(\begin{smallmatrix}v_{11}~&v_{12}\\
v_{21}~&v_{22}\end{smallmatrix}\right)\|_{\rm HS}^2~ {\rm(HS~ represents~ the~ Hilbert-Schmidt~ norm)}
\end{align*}
Consequently, $\sup_{\|\alpha\|=\|\beta\|=1}\| \rho_{\V}(p^{(\alpha,
\beta)}_\mathbf A) \|< \|\rho_{\V}^{(2)}(P_{\mathbf A})\|_{\rm op}$ if $V_1$ and $V_2$ are linearly independent.

\end{thm}
\begin{proof} By the definition of $\rho_{\V}$ we have
\begin{align*}
\sup_{\|\alpha\|=\|\beta\|=1}\|\rho_{\V}(p^{(\alpha
,\beta)}_\mathbf A)\|^2&=\sup_{\|\alpha\|=\|\beta\|=\|u\|=\|v\|=1}
 |\langle
A_{1}\alpha, \beta\rangle \langle V_1u, v\rangle+ \langle
A_{2}\alpha, \beta\rangle \langle V_2u,
v\rangle|^2\\&=\sup_{\|\alpha\|=\|\beta\|=\|u\|=1}|\alpha_1(v_{11}u_1+v_{12}u_2)+\alpha_2(v_{21}u_1+v_{22}u_2)|^2
|\beta_1|^2\\
&=\sup_{\|\alpha\|=\|u\|=1}|\alpha_1(v_{11}u_1+v_{12}u_2)+\alpha_2(v_{21}u_1+v_{22}u_2)|^2
\\&=\sup_{\|u\|=1}|v_{11}u_1+v_{12}u_2|^2+|v_{21}u_1+v_{22}u_2|^2\\&=\big\|\left(\begin{smallmatrix} v_{11}
& v_{12}\\
v_{21} & v_{22}\end{smallmatrix}\right)\big\|^{2}_{\rm
op}.\end{align*} On the other hand, we have  $$\|\rho_{\V}^{(2)}(P_{\mathbf
A})\|_{\rm op}^2=\|V_1\|^2+\|V_2\|^2 = \|\left(\begin{smallmatrix}v_{11}~&v_{12}\\
v_{21}~&v_{22}\end{smallmatrix}\right)\|^2_{\rm HS}.$$ 
 
If $V_1$ and $V_2$ are linearly independent $$\big\|\left(\begin{smallmatrix} v_{11}
& v_{12}\\
v_{21} & v_{22}\end{smallmatrix}\right)\big\|^{2}_{\rm op}
<\|\left(\begin{smallmatrix}v_{11}~&v_{12}\\
v_{21}~&v_{22}\end{smallmatrix}\right)\|^2_{\rm HS}$$ and we have
$$\sup_{\|\alpha\|=\|\beta\|=1}\| \rho_{\V}(p^{(\alpha,
\beta)}_\mathbf A) \|< \|\rho_{\V}^{(2)}(P_{\mathbf A})\|_{\rm op}.$$
\end{proof}

Now choose $V_1=\left(\begin{matrix}1 & 0\end{matrix}\right)$ and $V_2=\left(\begin{matrix}0 & 1\end{matrix}\right)$. From Lemma \ref{lemma1} and Theorem \ref{thm1} it follows that $\rho_\V$ is contractive but $\|\rho_\V^{(2)}(P_\A)\|=\sqrt{2}$.
\end{ex}

\section{Unitary Equivalence and Linear Equivalence}\label{unit-equiv}
If $U$  and $W$ are $2\times 2$  unitary matrices and $\widetilde{\A} = (U A_1 W, U A_2 W),$  then
$$\|(z_1,z_2)\|_{\A} = \|z_1 A_1 + z_2 A_2\|_{\rm op}
= \|z_1 (U A_1 W) + z_2 (U A_2 W) \|_{\rm op} = \|(z_1,z_2)\|_{\widetilde{\A}}.$$
There are, therefore, various choices of the matrix pair $(A_1,A_2)$ related as above which give rise to the same norm.  We use this freedom to ensure that $A_1$ is diagonal.
Consider the invertible linear transformation $(\tilde{z}_1,\tilde{z}_2)\mapsto (z_1,z_2)$ on $\mathbb C^2$ defined as follows:  

\noindent For $\tilde{\bf z}=(\tilde{z}_1,\tilde{z}_2)$ in $\mathbb{C}^2,$  let
\begin{eqnarray*}
z_1 &=& p \tilde{z}_1 + q \tilde{z}_2\\
z_2 &=& r \tilde{z}_1 + s \tilde{z}_2,
\end{eqnarray*}
where $p,q,r,s \in \mathbb C$.
Then $$\|(z_1,z_2)\|_\A = \|(\tilde{z}_1,\tilde{z}_2)\|_{\widetilde{\A}},$$
where $\widetilde{\A}$ is related to $\A$ as follows:
\begin{eqnarray*}
\widetilde{A}_1 &=& p A_1 + r A_2\\
\widetilde{A}_2 &=& q A_1 + s A_2.
\end{eqnarray*}
More concisely, if $T$ is the linear transformation above on $\mathbb C^2$, then 
$$\|T\tilde{\mathbf{z}}\|_\A = \|\tilde{\mathbf{z}}\|_{\A(T\otimes I)}.$$
In particular $T$ maps $\Omega_{\widetilde{\A}}$ onto $\Omega_{\A}$.

\begin{lem}\label{lemma2} For $k=1,2,\ldots,$ the contractivity  of the linear maps $L^{(k)}_\V$ defined on \\$(\mathbb C^2\otimes\mathcal{M}_k, \|\cdot\|_{\widetilde{\A},k}^*)$  determine the  contractivity  of the linear maps $L^{(k)}_{\widetilde{\V}}$ 
defined on $(\mathbb C^2\otimes \mathcal{M}_k, \|\cdot\|_{{\A},k}^*)$ and conversely, where  $\widetilde{\A} = \A(T\otimes I)$ and $\widetilde{\V} = (T\otimes I)\V.$
\end{lem}
\begin{proof}
For $k=1,2,\ldots,$ we have to show that 
$$\|L^{(k)}_{\V}\|_{(\mathbb C^2 \otimes \mathcal
M_{k},\|\cdot\|_{\widetilde{\A}, k}^*) \rightarrow (\mathcal
M_{k}\otimes \mathcal M_{p,q} , \|\cdot\|_{\rm op})}\leq1 \iff \|L^{(k)}_{\widetilde{\V}}\|_{(\mathbb C^2 \otimes \mathcal
M_{k}, \|\cdot\|_{{\A}, k}^*)\rightarrow (\mathcal M_{k}\otimes
\mathcal M_{p,q} , \|\cdot\|_{\rm op})}\leq 1.$$

We prove this result for the case $k=1$, that is, for the map $L_\V$.  The proof for the general case is similar.

Consider the bijection between the spaces $\{f\in{\rm Hol}(\Omega_\A,\mathbb D), f(0)=0\}$ and \\
$\{\widetilde{f}\in{\rm Hol}(\Omega_{\widetilde{\A}},\mathbb D), \widetilde{f}(0)=0\}$ defined as follows:
$$f\mapsto \widetilde{f}=f\circ T, ~~~~~\widetilde{f}\mapsto f =\widetilde{f}\circ T^{-1}$$
Using this bijection
\begin{align*}
\|\,L_{\V}\,\|_{(\mathbb C^2, \|\cdot \|_{\widetilde{\A}}^* )\rightarrow (\mathcal M_{p,q}, \|\cdot\|_{\rm op})}\leq 1
&\iff\sup_{\widetilde{f}} \{\|D\widetilde{f} (0) \cdot \V\|_{\rm op}: 
\widetilde{f}\in{\rm Hol}(\Omega_{\widetilde{\A}},\mathbb D), \widetilde{f}(0)=0\}\leq 1\\
&\iff\sup_{f} \{\|D(f\circ T) (0) \cdot \V\|_{\rm op}: 
f\in{\rm Hol}(\Omega_{\A},\mathbb D), f(0)=0\}\leq 1\\
&\iff \sup_{f} \{\|Df(0)~ T \cdot \V\|_{\rm op}: 
f\in{\rm Hol}(\Omega_{\A},\mathbb D), f(0)=0\}\leq 1\\
&\iff\sup_{f} \{\|Df(0) \cdot (T \otimes I) \V\|_{\rm op}: 
f\in{\rm Hol}(\Omega_{\A},\mathbb D), f(0)=0\}\leq 1\\
&\iff\|\,L_{(T\otimes I)\V}\,\|_{(\mathbb C^2, \|\cdot \|_{\A}^*)\rightarrow (\mathcal M_{p,q}, \|\cdot\|_{\rm
op})}\leq 1
\end{align*}
In the above, $Df$ is a row vector, $T$ is a $2\times 2$ matrix and by an expression of the form $X\cdot Y$ we mean $\sum_{i=1}^2 X_i Y_i$. \end{proof}

It follows that, in our study of the existence of contractive homomorphisms which are not completely contractive, two sets of matrices $\A = (A_1,A_2)$ and $\widetilde{\A} = (\widetilde{A}_1, \widetilde{A}_2)$ which are related through linear combinations as above yield the same result.  We can, therefore, restrict our attention to a subcollection of matrices.

Since $A_1$ has already been chosen to be diagonal, we consider transformations as above with $r=0$ to preserve the diagonal structure of $A_1$.  By choosing the parameters $p,q,s$ suitably we can ensure that one diagonal entry of $A_1$ is $1$ and the diagonal entries of $A_2$ are $1$ and $0$.  By further conjugating with a diagonal unitary and a permutation matrix it follows that we need to consider only the following three families of matrices:

\begin{center}

\begin{table}[ht]\label{Table1}
\caption{Cases modulo unitary and linear equivalence}
\begin{tabular}{ |c|c| }
\hline
$A_1$ & $A_2$ \\ \hline \hline

{$\begin{array}{c}
\\
\left(\begin{array}{cc}
1 &  0 \\
0 & d \\
\end{array}\right)\\
\\
\end{array} d\in \mathbb C$}  & $\begin{array}{c}\\
\left(\begin{array}{cc}
1 &  b \\
c & 0 
\end{array}\right)c\in \mathbb C,b\in\mathbb R_+
\\
\\
\end{array}$  \\ \hline

{$\begin{array}{c}
\\
\left(\begin{array}{cc}
d &  0 \\
0 & 1 \\
\end{array}\right)\\
\\
\end{array} d\in \mathbb C$} 
 & $\begin{array}{c} 
 \\
 \left(\begin{array}{cc}
1 &  b \\
c & 0 \\
\end{array}\right)c\in \mathbb C,b\in \mathbb R_+
\\
\\
\end{array}$ \\  \hline
{$\begin{array}{c}
\\
\left(\begin{array}{cc}
1 &  0 \\
0 & d \\
\end{array}\right)\\
\\
\end{array} d\in \mathbb C$} 
 & $\begin{array}{c} 
 \\
 \left(\begin{array}{cc}
0 &  b \\
c & 0 \\
\end{array}\right)c\in \mathbb C,b\in \mathbb R_+
\\
\\
\end{array}$ \\  \hline

\end{tabular}
\end{table}
\end{center}
In the above, $\mathbb{R}_+$ represents the set of non-negative real numbers.


\subsection{Simultaneously Diagonalizable Case}

For the study of contractivity and complete contractivity in this situation we consider two possibilities.  The first when $A_1$ and $A_2$ are simultaneously diagonalizable and the second when they are not.  The simultaneously diagonalizable case reduces to the case of the bi-disc where we know that any contractive homomorphism is completely contractive.  In all the other cases (when $A_1$ and $A_2$ are not simultaneously diagonalizable) we show that there exists a contractive homomorphism which is not completely contractive. 

Consider first the case when $A_1$ and $A_2$  are simultaneously diagonalizable. Based on the discussion of linear equivalence above we need to study only the following possibilities:
\begin{center}
\begin{table}[ht]\label{Table2}
\caption{Simultaneously diagonalizable cases}
\begin{tabular}{ |c|c| }
\hline
$A_1$ & $A_2$ \\ \hline \hline

{$\begin{array}{c}
\\
\left(\begin{array}{cc}
1 &  0 \\
0 & d \\
\end{array}\right)\\
\\
\end{array} d\in \mathbb C$}  & $\begin{array}{c}\\
\left(\begin{array}{cc}
1 &  0 \\
0 & 0 
\end{array}\right)
\\
\\
\end{array}$  \\ \hline

{$\begin{array}{c}
\\
\left(\begin{array}{cc}
d &  0 \\
0 & 1 \\
\end{array}\right)\\
\\
\end{array} d\in \mathbb C$} 
 & $\begin{array}{c} 
 \\
 \left(\begin{array}{cc}
1 &  0 \\
0 & 0 \\
\end{array}\right)
\\
\end{array}$ \\  \hline
\end{tabular}
\end{table}
\end{center}

Applying linear transformations as before, both cases can be reduced to $\mathbf A=\left(\left (
\begin{smallmatrix}
1 & 0   \\
0  &  0
\end{smallmatrix}\right ) ,
\left ( \begin{smallmatrix}
0&  0\\
0  &  1
\end{smallmatrix}\right )\right)$ which represents the bi-disc.
As mentioned earlier, it is known that any contractive homomorphism is completely contractive in this case.
We now study the situation when $A_1$ and $A_2$ are not simultaneously diagonalizable.

\section{Contractivity, Complete Contractivity and Operator Space Structures}\label{opspace}

We recall some notions about operator spaces which are relevant to our purpose.

\begin{defn} (cf. \cite[Chapter 13, 14]{paulsen}) \label{def1}
An abstract operator space is a linear space $X$ together
with a family of norms $\|\cdot\,\|_{k}$ defined on $\mathcal
M_{k}(X)$, $k=1,2, 3, \ldots ,$ where $\|\,\cdot\|_1$ is simply a
norm on the linear space $X$. These norms are required to
satisfy the following compatibility conditions:
\begin{enumerate}
\item $\|T \oplus S\|_{p+q}=\max\{\|T\|_p, \|S\|_q\}$ and \item
$\|ASB\|_{p} \leq \|A\|_{\rm op}\|S\|_q\|B\|_{\rm op}$
\end{enumerate}
for all $S \in \mathcal M_q(X),~ T \in \mathcal M_p(X)$ and $A \in \mathcal M_{p, q}(\mathbb C), B \in \mathcal
M_{q, p}(\mathbb C).$
\end{defn}
Two such operator spaces $(X, \|\cdot\|_{ k})$ and
$(Y, \|\cdot\|_{ k})$  are said to be completely isometric
if there is a linear bijection $T : X \to Y$ such
that $T \otimes I_{k}:(\mathcal M_{k}(X), \|\cdot\|_{
k})\to (\mathcal M_{k}(Y), \|\cdot\|_{ k})$ is an isometry
for every $k \in \mathbb N.$ Here we have identified $\mathcal
M_k(X)$ with $X \otimes \mathcal M_k$ in the usual
manner. We note that a normed linear space $(X,
\|\,\cdot\|)$ admits an operator space structure if and only if
there is an isometric embedding of it into the algebra of
operators $\mathcal B(\mathcal H)$ on some Hilbert space $\mathcal
H$. This is the well-known theorem of Ruan (cf. \cite{pisier}).
 
We recall here the notions of MIN and
MAX operator spaces and a measure of their distance, $\alpha(X)$, following \cite[Chapter 14]{paulsen}.
\begin{defn}
The MIN operator structure $MIN(X)$ on a (finite
dimensional) normed linear space $X$ is obtained by isometrically
embedding $ X $ in the $C^*$ algebra $C\big ( (X^*)_1 \big ),$ of continuous functions on the unit ball $(X^*)_1$ of the dual space. Thus for $(\!(v_{ij})\!)$ in $\mathcal
M_k(X)$, we set
$$\|(\!(v_{ij})\!)\|_{MIN}=\|(\!(\widehat{v_{ij}})\!)\|=\sup\{\|(\!(f(v_{ij}))\!)\|:f
\in (X^*)_1\},$$ where the norm of a scalar matrix
$(\!(f(v_{ij}))\!)$ in $\mathcal M_k$ is the operator norm.
\end{defn}
For an arbitrary $k \times k$ matrix over $X,$ we simply
write $\|(\!(v_{ij})\!)\|_{MIN(X)}$ to denote its norm in
$\mathcal M_{k}(X).$ This is the minimal way in which we
represent the normed space as an operator space. There is also a `maximal' representation which is denoted $MAX(X).$
\begin{defn}
The operator space $MAX(X)$ is defined by setting
$$\|(\!(v_{ij})\!)\|_{MAX}=\sup\{\|(\!(T(v_{ij}))\!)\|:T:X \rightarrow B(\mathcal H) \},$$
 and the supremum is taken over all isometries $T$
and all Hilbert spaces $\mathcal H.$
\end{defn}
Every operator space structure on a normed linear space $X$ `lies between' $MIN(X)$ and $MAX(X)$. The extent to which the two operator space structures $MIN(X)$ and $MAX(X)$ differ is characterized by the constant
$\alpha(X)$ introduced by Paulsen (cf.\cite[Chapter 14]{paulsen}), which we recall below.
\begin{defn}
The constant $\alpha(X)$ is defined as
$$\alpha(X)=\sup\{\|(\!(v_{ij})\!)\|_{MAX}:\|(\!(v_{ij})\!)\|_{MIN}
\leq 1,\,\, (\!(v_{ij})\!) \in \mathcal M_k(X), k\in
\mathbb N\}.$$
\end{defn}
Thus $\alpha(X)=1$ if and only if the identity map is a
complete isometry from $MIN(X)$ to $MAX(X).$
Equivalently, we conclude that there exists a unique operator space
structure on $X$ whenever $\alpha(X)$ is $1$.
Therefore, those normed linear spaces for which $\alpha(X)=1$ are rather special. Unfortunately, there aren't too many of
them! The familiar examples are 
$(\mathbb C^2,\|\cdot\|_{\infty}),$ and consequently $\mathbb
C^2$ with the $\ell_1$ norm. 
It  is pointed out in \cite[pp. 76]{pisier}) that $\alpha(X)>1$ for $\dim(X) \geq 3$, refining an earlier result of Paulsen that $\alpha(X)>1$ whenever $\dim(X) \geq 5$.  This leaves the question open for normed linear spaces  whose
dimension is $2$. 

Returning to the space $(\mathbb C^2, \|\cdot \|_{\mathbf A})$ with $\|(z_1,z_2)\|_\A=\|z_1 A_1 + z_2 A_2\|_{\rm op},$ we show below that $\alpha(\Omega_{\mathbf A}) > 1$ in a large number of cases. From \cite[Theorem 4.2]{vern}, it therefore follows that, in all these cases, there must exist a contractive homomorphism of $\mathcal O(\Omega_{\mathbf A})$ into the algebra $B(\mathcal H)$ which is not completely contractive.  In the remaining cases, the existence of a contractive homomorphism which is not completely contractive is established by a careful study of certain extremal problems.
%
%

The norm $\|(z_1,z_2)\|_\A=\|z_1 A_1 + z_2 A_2\|_{\rm op}$ defines a  natural isometric embedding into $\mathcal {M}_2(\mathbb C)$ given by $(z_1,z_2) \mapsto z_1 A_1 + z_2 A_2$.  However, note that
$$\|(z_1,z_2)\|_\A=\|z_1 A_1 + z_2 A_2\|_{\rm op} = \|z_1 A_1^{\rm t} + z_2 A_2^{\rm t}\|_{\rm op} = \|(z_1,z_2)\|_{\A^{\rm t}}.$$
We, therefore, get another isometric embedding into  $\mathcal {M}_2(\mathbb C)$ given by $(z_1,z_2) \mapsto z_1 A_1^{\rm t} + z_2 A_2^{\rm t}$.

In a variety of cases the operator spaces determined by these two embeddings are distinct and the parameter $\alpha > 1$ in these cases. 
Therefore, the existence of contractive homomorphisms which are not completely contractive is established in these cases. We present the details below.

Recall the map $P_{\A}$ defined earlier by $P_{\A}(z_1,z_2) = z_1 A_1 + z_2 A_2$.  Let $P_{\A}^{(2)}=P_{\A} \otimes I_2$. For the three families of matrices $\A=(A_1,A_2)$ characterized in Table 1 we show that $\A$ and $\A^{\rm t}$ define distinct operator space structures unless $|d|=1$ or $b=|c|$.

\begin{thm}\label{3}
Let  $Z_1=\left(\begin{smallmatrix}1  & 0\\
0 & 0\end{smallmatrix}\right)$ and $Z_2=\left(\begin{smallmatrix}0  & 1\\
0 & 0\end{smallmatrix}\right)$. If $|d|\neq 1$ and $b\neq|c|$ then $\|P_{\mathbf
A}^{(2)}(Z_1,Z_2)\|_{\rm op}\neq\|P_{\mathbf A^{\rm t}}^{(2)}(Z_1,Z_2)\|_{\rm
op}$.
\end{thm}
\begin{proof}

We illustrate the proof for the case $A_1=\left(\begin{smallmatrix}1  & 0\\
0 & d\end{smallmatrix}\right), A_2=\left(\begin{smallmatrix}1  & b\\
c & 0\end{smallmatrix}\right)$.  The other cases can be proved similarly.

For this case
\begin{align}\|P_{\mathbf
A}^{(2)}(Z_1,Z_2)\|_{\rm op}^2&=\left\|\left(\begin{smallmatrix}(Z_1+Z_2)~~ &~~ bZ_2\\
cZ_2 & d Z_1 \end{smallmatrix}\right)\left(\begin{smallmatrix}(Z_1+Z_2)^{*}~~ &~~ \bar{c}Z_2^{*}\\
bZ_2^{*} & \bar{d}Z_1^{*}
\end{smallmatrix}\right)\right\|_{\rm op}\nonumber\\&=
\left\|\left(\begin{smallmatrix}(Z_1+Z_2)(Z_1+Z_2)^*+b^2 Z_2Z_2^{*}~~ &~~
\bar{c}(Z_1+Z_2)Z_2^{*}
+b \bar{d}Z_2Z_1^{*}\\
cZ_2(Z_1 + Z_2)^{*}+bdZ_1Z_2^{*} & |c|^2Z_2Z_2^{*}
+|d|^2Z_1Z_1^{*}
\end{smallmatrix}\right)\right\|_{\rm op}.\end{align}
 Similarly we have \begin{align}\|P_{\mathbf
A^{\rm t}}^{(2)}(Z_1,Z_2)\|_{\rm op}^2&=
\left\|\left(\begin{smallmatrix}(Z_1+Z_2)(Z_1+Z_2)^*+|c|^2 Z_2Z_2^{*}~~ &~~
b(Z_1+Z_2)Z_2^{*}
+c \bar{d}Z_2Z_1^{*}\\
bZ_2(Z_1 + Z_2)^{*}+\bar{c}dZ_1Z_2^{*} & b^2Z_2Z_2^{*}
+|d|^2Z_1Z_1^{*}
\end{smallmatrix}\right)\right\|_{\rm op}.\end{align}
Assume  $\|P_{\mathbf A}^{(2)}(Z_1,Z_2)\|_{\rm
op}^2=\|P_{\mathbf A^{\rm t}}^{(2)}(Z_1,Z_2)\|_{\rm op}^2.$ Using the form of $(Z_1,Z_2)$ this is equivalent to 
$$\left\|\left(\begin{smallmatrix}2+b^2  & \bar{c}\\
c & |c|^2+|d|^2\end{smallmatrix}\right)\right\|_{\rm op} = \left\|\left(\begin{smallmatrix}2+|c|^2  & b\\
b & b^2+|d|^2\end{smallmatrix}\right)\right\|_{\rm op}$$
i.e. $(b^2-|c|^2)(1-|d|^2)=0$ (note that the matrices on the left and right have the same trace), from which the result follows.
\end{proof}

Since $\alpha(\Omega_{\mathbf A})=1$ if and only if the two operator spaces MIN($\Omega_{\mathbf A}$) and MAX($\Omega_{\mathbf A}$) are completely isometric, it follows from the Theorem we have just proved that if $|d|\neq 1$ and $b\neq |c|,$ then $\alpha(X) > 1.$  Consequently, there exists 
a contractive homomorphism of $\mathcal O(\Omega_\A)$ into $B(\mathcal H),$ which is not completely contractive.
%

\begin{ex}(Euclidean Ball) The Euclidean ball $\mathbb{B}^2$ is characterized by $A_1=\left(\begin{smallmatrix}1  & 0\\
0 & 0\end{smallmatrix}\right), A_2=\left(\begin{smallmatrix}0  & 1\\
0 & 0\end{smallmatrix}\right)$. So, in Theorem \ref{3}, we have $|d|\neq 1$ and $b\neq |c|$.  Hence $\A$ and $\A^{\rm t}$ give rise to distinct operator space structures and, consequently, there exists a contractive homomorphism which is not completely contractive.
\end{ex}


\section{Cases not Amenable to the Operator Space Method}

Theorem \ref{3} shows that there is a contractive homomorphism which is not completely contractive for all the choices of $(A_1,A_2)$ listed in Table 1 except when $|d|=1$ or $b=|c|$.  We are, therefore, left with the following families of $(A_1,A_2)$ to be considered:

\begin{center}

\begin{table}[ht]\label{Table3}
\caption{Cases not covered by the operator space approach}
\begin{tabular}{ |c|c|c| }
\hline
&$A_1$ & $A_2$ \\ \hline \hline

(i)&{$\begin{array}{c}
\\
\left(\begin{array}{cc}
1 &  0 \\
0 & e^{i\theta} \\
\end{array}\right)\\
\\
\end{array} \theta\in \mathbb R$}  & $\begin{array}{c}\\
\left(\begin{array}{cc}
1 &  b \\
c & 0 
\end{array}\right)c\in \mathbb C,b\in\mathbb R_+
\\
\\
\end{array}$  \\ \hline

(ii)&{$\begin{array}{c}
\\
\left(\begin{array}{cc}
1 &  0 \\
0 & e^{i\theta} \\
\end{array}\right)\\
\\
\end{array} \theta\in \mathbb R$}  & $\begin{array}{c}\\
\left(\begin{array}{cc}
0 &  b \\
c & 0 
\end{array}\right)c\in \mathbb C,b\in\mathbb R_+
\\
\\
\end{array}$  \\ \hline

(iii)&{$\begin{array}{c}
\\
\left(\begin{array}{cc}
e^{i\theta} &  0 \\
0 & 1 \\
\end{array}\right)\\
\\
\end{array} \theta\in \mathbb R$} 
 & $\begin{array}{c} 
 \\
 \left(\begin{array}{cc}
1 &  b \\
c & 0 \\
\end{array}\right)c\in \mathbb C,b\in \mathbb R_+
\\
\\
\end{array}$ \\  \hline

(iv)&{$\begin{array}{c}
\\
\left(\begin{array}{cc}
1 &  0 \\
0 & d \\
\end{array}\right)\\
\\
\end{array} d\in \mathbb C$} 
 & $\begin{array}{c} 
 \\
 \left(\begin{array}{cc}
1 &  |c| \\
c & 0 \\
\end{array}\right)c\in \mathbb C
\\
\\
\end{array}$ \\  \hline

(v)&{$\begin{array}{c}
\\
\left(\begin{array}{cc}
1 &  0 \\
0 & d \\
\end{array}\right)\\
\\
\end{array} d\in \mathbb C$}  & $\begin{array}{c}\\
\left(\begin{array}{cc}
0 &  |c| \\
c & 0 
\end{array}\right)c\in \mathbb C
\\
\\
\end{array}$  \\ \hline

(vi)&{$\begin{array}{c}
\\
\left(\begin{array}{cc}
d &  0 \\
0 & 1 \\
\end{array}\right)\\
\\
\end{array} d\in \mathbb C$} 
 & $\begin{array}{c} 
 \\
 \left(\begin{array}{cc}
1 &  |c| \\
c & 0 \\
\end{array}\right)c\in \mathbb C
\\
\\
\end{array}$ \\  \hline

\end{tabular}
\end{table}
\end{center}

These six families are not disjoint and have been classified as such on the basis of the method of proof used.

\subsection{Dual norm method}

We first consider a special case of type (ii) in Table 3 with $A_1=\left(\begin{smallmatrix}1  & 0\\
0 & 1\end{smallmatrix}\right), A_2=\left(\begin{smallmatrix}0  & 1\\
0 & 0\end{smallmatrix}\right)$.  Although this case is covered by the more general method to be outlined later we present an alternate, interesting procedure for this example since it is possible to explicitly calculate the dual norm $\|\cdot\|_{\A}^*$ in this case. Equipped with the information about the dual norm we can directly construct a pair $\V=(V_1,V_2)$ such that $\|L_{\V}\| \leq 1$ and $\|L_{\V}^{(2)}(P_{\A})\|>1$.

Note that in this case
$$\|(z_1,z_2)\|_{\A} = \frac{|z_2|+\sqrt{|z_2|^2 + 4|z_1|^2}}{2}$$
and the unit ball 
$$\Omega_\A = \{(z_1,z_2):|z_1|^2 + |z_2|<1\}.$$

\begin{lem}\label{4}
Let $A_1=\left(\begin{smallmatrix}1  & 0\\
0 & 1\end{smallmatrix}\right), A_2=\left(\begin{smallmatrix}0  & 1\\
0 & 0\end{smallmatrix}\right)$. If $(\omega_1,\omega_2)\in (\mathbb C^2,\|\cdot\|_{\A}^*)$ then the dual norm
\[ \|(\omega_1,\omega_2)\|_{\A}^* = \left\{ \begin{array}{ll}
        \frac{|\omega_1|^2+4|\omega_2|^2}{4|\omega_2|}  & \mbox{if $|\omega_2|\geq \frac{|\omega_1|}{2} $};\\
        |\omega_1| & \mbox{if $|\omega_2|\leq \frac{|\omega_1|}{2}$}.\end{array} \right. \] 
\end{lem}
\begin{proof}
Let $f_{\omega_1,\omega_2}$ be the linear functional on $(\mathbb C^2, \|\cdot \|_\A)$ defined by 
$$f_{\omega_1,\omega_2}(z_1,z_2) = \omega_1 z_1 +\omega_2 z_2.$$
Then
\begin{align*}
\|(\omega_1,\omega_2)\|_{\A}^* &= \sup_{(z_1,z_2)\in \Omega_\A}|f_{\omega_1,\omega_2}(z_1,z_2)|\\
&=\sup_{|z_2| \leq 1 - |z_1|^2}|\omega_1 z_1 +\omega_2 z_2|\\
&=\sup_{|z_2| \leq 1 - |z_1|^2}(|\omega_1||z_1|+|\omega_2||z_2|)\\
&=\sup_{|z_1| \leq 1}\big(|\omega_1||z_1|+|\omega_2|(1-|z_1|^2)\big).
\end{align*}

If $|\omega_2|\geq \frac{|\omega_1|}{2} $ the expression on the right attains its maximum at $|z_1| = \frac{|\omega_1|}{2|\omega_2|} \leq 1$ and the maximum value is $\frac{|\omega_1|^2+4|\omega_2|^2}{4|\omega_2|}$.

If $|\omega_2|\leq \frac{|\omega_1|}{2} $ the expression on the right is monotonic in $|z_1|$ and the maximum is attained at $|z_1|=1$.  The maximum value in this case is $|\omega_1|$.

\end{proof}

\begin{thm}\label{5}
Let  $A_1=\left(\begin{smallmatrix}1  & 0\\
0 & 1\end{smallmatrix}\right), A_2=\left(\begin{smallmatrix}0  & 1\\
0 & 0\end{smallmatrix}\right)$ and 
 $V_1=\left(\begin{matrix}\frac{1}{\sqrt{2}}  & 0 \end{matrix}\right), V_2=\left(\begin{matrix}0  & 1\end{matrix}\right).$
  Then 
\begin{align*}
&(i)~~\|L_{\V}\|_{(\mathbb C^2,\|\cdot\|_{\A}^*)\rightarrow(\mathbb C^2, \|\cdot\|_2)}=1\\
&(ii)~~\|L_{\V}^{(2)}(P_\A)\|=\sqrt{\frac{3}{2}}.
\end{align*}
Consequently $\rho_\V$, for this choice of $\V = (V_1,V_2)$, is contractive on $\mathcal O(\Omega_\A)$ but not completely contractive.
\end{thm}

\begin{proof}[Proof of (i)]

\begin{align*}
\|L_{\V}\|_{(\mathbb C^2,\|\cdot\|_{\A}^*)\rightarrow(\mathbb C^2, \|\cdot\|_2)}^2 &= \sup_{\|(\omega_1,\omega_2)\|_{\A}^*=1}\|\omega_1 V_1 + \omega_2 V_2\|_2^2\\
&=\sup_{\|(\omega_1,\omega_2)\|_{\A}^*=1}\Big(\frac{|\omega_1|^2}{2} + |\omega_2 |^2\Big).
\end{align*}

We now consider two cases:

Case (a): $|\omega_2|\geq \frac{|\omega_1|}{2} $ and $1=\|(\omega_1,\omega_2)\|_{\A}^*=\frac{|\omega_1|^2+4|\omega_2|^2}{4|\omega_2|}~~$ from Lemma \ref{4}.

These two constraints together can be seen to be equivalent to the constraints $\frac{1}{2}\leq |\omega_2| \leq 1$ and $|\omega_1|^2=4|\omega_2|(1-|\omega_2|)$.

Hence the supremum above for this range of $(\omega_1,\omega_2)$ is given by 
$$\sup_{\frac{1}{2}\leq |\omega_2| \leq 1}|\omega_2|\big(2-|\omega_2|\big)=1.$$

Case (b): $|\omega_2|\leq \frac{|\omega_1|}{2} $ and $1=\|(\omega_1,\omega_2)\|_{\A}^*=|\omega_1|~~$ from Lemma \ref{4}.

The supremum for this range of $(\omega_1,\omega_2)$ is given by 
$$\sup_{|\omega_2| \leq \frac{1}{2}}\big(\frac{1}{2}+|\omega_2|^2\big)=\frac{3}{4}.$$

Taking the larger of the supremums in Case (a) and Case (b) we get that $\|L_\V\|=1$.
\end{proof}
\begin{proof}[Proof of (ii)]
\begin{align*}
\|L_\V^{(2)}(P_\A)\|^2&=\|A_1\otimes V_1 + A_2 \otimes V_2\|^2\\
&=\Big\|\left( \begin{matrix}\frac{1}{\sqrt{2}}A_1 & A_2 \end{matrix}\right)\Big\|^2\\
&=\Bigg\|\left( \begin{matrix}\frac{1}{\sqrt{2}}A_1 & A_2 \end{matrix}\right)\left(\begin{matrix}\frac{1}{\sqrt{2}}A_1^*\\ \\ A_2^*\end{matrix}\right)\Bigg\|\\
&=\bigg\|\left(\begin{matrix}\frac{3}{2} & 0\\
0 & \frac{1}{2}\end{matrix}\right)\bigg\|~~\mbox{using the form of}~~ A_1,A_2\\
&=\frac{3}{2}.
\end{align*}
\end{proof}
\vspace{-0.175in}
\subsection{General cases not amenable to the operator space method}

The various families of $(A_1,A_2)$ listed in Table 3 require a case by case analysis to show that there is a contractive homomorpism which is not completely contractive.  We first present a general outline of the method used.  

We choose the pair $\V=(V_1,V_2)$ to be of the form $V_1=\left(\begin{matrix}u&0 \end{matrix}\right), V_2=\left(\begin{matrix}0&v \end{matrix}\right),~~u,v\in \mathbb R_+$.  $L_\V:(\mathbb C^2,\|\cdot\|_{\A}^*) \rightarrow (\mathbb C^2, \|\cdot\|_2)$ then becomes the linear map $(z_1,z_2) \mapsto (z_1u,z_2v)$.

We show, in each case, that by a suitable choice of $u$ and $v$ we can ensure that $L_\V$ is contractive while $\|L_{\V}^{(2)}(P_\A)\|>1$ although $\|P_\A\|=1$ by definition.

We list the contractivity conditions (see Propositions \ref{14} and \ref{15} for details).

(a) $L_\V$ is contractive if and only if the following two conditions are satisfied:
\begin{enumerate}
\item[(i)] $u\leq \frac{1}{\|A_1^{*}\|}$ or $v\leq \frac{1}{\|A_2^{*}\|}$ and
\item[(ii)] 
\begin{equation}\label{6}
\inf_{\beta\in\mathbb C^2,\|\beta\|=1}\Big\{1-u^2\|A_{1}^*\beta\|^2
-v^2\|A_{2}^*\beta\|^2+
u^2v^2\Big(\|A_{1}^*\beta\|^2\|A_{2}^*\beta\|^2 - |\left\langle
A_1A_{2}^*\beta, \beta\right\rangle |^2\Big)\Big\}\geq 0.
\end{equation}

\end{enumerate}

(b) $\|L_{\V}^{(2)}(P_\A)\| \leq1$ if and only if

\begin{equation}\label{7}
\inf_{\beta\in\mathbb C^2,\|\beta\|=1}\Big\{1-u^2\|A_{1}^*\beta\|^2
-v^2\|A_{2}^*\beta\|^2\Big\}\geq 0.
\end{equation}

Note that the term in parenthesis in (\ref{6}) is non-negative by the Schwarz inequality and that the expression (\ref{7}) is the same as the first three terms in (\ref{6}). 

We show that, in each case, we can choose $(u, v)$ such that the infimum in (\ref{6}) is exactly 0.  Also that this infimum is attained at $\beta = \beta_0$ such that the term in parenthesis in (\ref{6}) is positive (that is, the Schwarz inequality referred to above is a strict inequality at $\beta_0$).  It then follows that the expression in braces in (\ref{7}) is negative when $\beta = \beta_0$ and, consequently, the infimum in (\ref{7}) is negative.  Taken together it follows that $L_\V$ (and consequently $\rho_\V$) is contractive but $\|L_\V^{(2)}(P_\A)\| > 1$ and, as a result, $\rho_\V^{(2)}$ is not contractive.


Let $\eta^{(i)}, i=1,2,$ be the vectors such that $A_1^*\eta^{(i)}$ and $A_2^* \eta^{(i)}$ are linearly dependent.  That is, the term in parenthesis in (\ref{6}) vanishes when $\beta = \eta^{(i)}$.

We now provide the details of the argument which proceeds in two steps.

{\bf Step 1:} Show that there are certain ranges of the parameters $(u,v)$ such that the infimum in (\ref{6}) is not attained at $\eta^{(1)}$ or $\eta^{(2)}$ for those values of $(u,v)$.

Let
$$g_{u,v}(\beta)=1-u^2\|A_{1}^*\beta\|^2
-v^2\|A_{2}^*\beta\|^2+
u^2v^2\Big(\|A_{1}^*\beta\|^2\|A_{2}^*\beta\|^2 - |\left\langle
A_1A_{2}^*\beta, \beta\right\rangle |^2\Big).$$

We need to show that there exists $\beta$ such that 
\begin{equation*}\label{1A}
g_{u,v}(\beta) < g_{u,v}(\eta^{(i)}),~~i=1,2,
\end{equation*}
when $(u,v)$ take values in a range of interest.  That is, 
we need to find $\beta$ such that 
\begin{equation}\label{3A}
g_{u,v}(\eta^{(i)})- g_{u,v}(\beta)=
a_i(\beta)u^2+b_i(\beta)v^2-c(\beta)u^2v^2 > 0.
\end{equation}
Here
\begin{eqnarray} 
a_i(\beta)&=&\|A_1^*\beta\|^2-\|A_1^*\eta^{(i)}\|^2 \label{eq3}\\
b_i(\beta)&=&\|A_2^*\beta\|^2-\|A_2^*\eta^{(i)}\|^2 \nonumber\\
c(\beta)&=&\|A_{1}^*\beta\|^2\|A_{2}^*\beta\|^2 - |\left\langle
A_1A_{2}^*\beta, \beta\right\rangle |^2 \geq 0. \nonumber
\end{eqnarray}

Consider the functions 
$$f_i(u,v, \beta)=a_i(\beta)u^2+b_i(\beta)v^2-c(\beta)u^2v^2 ~~{\rm with}~~ c(\beta)\geq 0,~~ i=1,2.$$

The following result is evident from the nature of the functions $f_i(u,v, \beta)$.

\begin{lem}\label{18}
{\rm (i)} Assume $a_i(\beta) > 0$ for some fixed $\beta$ and $i=1,2$.  Then, given any $u_0>0$, there exists $v_0>0$ (depending on $u_0$) such that $f_i(u,v, \beta)>0$ in the region $u<u_0, v<\frac{v_0}{u_0} u,$  that is, inside the triangle with vertices $(0,0), (u_0,0)$ and $(u_0,v_0)$.
\begin{enumerate}
\item [(ii)] Assume $b_i(\beta) > 0$ for some fixed $\beta$ and $i=1,2$.  Then, given any $v_0>0$, there exists $u_0>0$ (depending on $v_0$) such that $f_i(u,v, \beta)>0$ in the region $v<v_0, u<\frac{u_0}{v_0} v$ , that is, inside the triangle with vertices $(0,0), (0,v_0)$ and $(u_0,v_0)$.
\item[]
\item [(iii)] If $f_i(u_0,v_0,\beta)>0$ then $f_i(t u_0,tv_0,\beta)>0$ for $0<t<1$.
\end{enumerate}
\end{lem}

We will show below that, in each of the six cases in Table 3, it is possible to ensure the positivity of $a_i(\beta),~i=1,2$ or $b_i(\beta),~i=1,2$ for some choice of $\beta$.  Consequently, it will follow that the inequality (\ref{1A}) will be true for that vector $\beta$ with $(u,v)$ in the region characterized in Lemma \ref{18} above.  Hence, for $(u,v)$ in this range, the infimum is not attained at $\eta^{(i)}, i=1,2$.

Consider first the cases (i), (ii) and (iii).

\begin{center}
\begin{tabular}{ |c|c|c| }
\hline
&$A_1$ & $A_2$ \\ \hline \hline
(i)&{$\begin{array}{c}
\\
\left(\begin{array}{cc}
1 &  0 \\
0 & e^{i\theta} \\
\end{array}\right)\\
\\
\end{array} \theta\in \mathbb R$}  & $\begin{array}{c}\\
\left(\begin{array}{cc}
1 &  b \\
c & 0 
\end{array}\right)c\in \mathbb C,b\in\mathbb R_+
\\
\\
\end{array}$  \\ \hline

(ii)&{$\begin{array}{c}
\\
\left(\begin{array}{cc}
1 &  0 \\
0 & e^{i\theta} \\
\end{array}\right)\\
\\
\end{array} \theta\in \mathbb R$}  & $\begin{array}{c}\\
\left(\begin{array}{cc}
0 &  b \\
c & 0 
\end{array}\right)c\in \mathbb C,b\in\mathbb R_+
\\
\\
\end{array}$  \\ \hline

(iii)&{$\begin{array}{c}
\\
\left(\begin{array}{cc}
e^{i\theta} &  0 \\
0 & 1 \\
\end{array}\right)\\
\\
\end{array} \theta\in \mathbb R$} 
 & $\begin{array}{c} 
 \\
 \left(\begin{array}{cc}
1 &  b \\
c & 0 \\
\end{array}\right)c\in \mathbb C,b\in \mathbb R_+
\\
\\
\end{array}$ \\  \hline

\end{tabular}
\end{center}

We use the unitary equivalence described in Section \ref{unit-equiv}. In cases (i) and (ii) multiply $A_1$ and $A_2$ on the left by the unitary matrix $\left(\begin{smallmatrix}1  & 0\\
0 & e^{-i\theta}\end{smallmatrix}\right)$ so that $A_1$ becomes the identity matrix. In case (iii) multiply $A_1$ and $A_2$ on the left by the unitary matrix $\left(\begin{smallmatrix}e^{-i\theta}  & 0\\
0 & 1\end{smallmatrix}\right)$ so that $A_1$ becomes the identity matrix.

Now conjugate $A_1$ and $A_2$ by the unitary which makes $A_2$ upper triangular so that cases (i),(ii) and (iii) reduce to the situation
$$A_1 = \left(\begin{smallmatrix}1  & 0\\
0 & 1\end{smallmatrix}\right) ~{\rm and }~ A_2=\left(\begin{smallmatrix}\mu  & \sigma\\
0 & \nu\end{smallmatrix}\right) ~{\rm with}~ |\mu|\geq |\nu|, \sigma\neq 0$$
In this case $a_i(\beta)=0$ for all $\beta$ but it is possible to choose $\beta$ such that $b_i(\beta)>0$.

$\eta^{(i)}$ satisfies the equation $(A_2^* - \lambda_i A_1^*)\eta^{(i)} =0$.  So in this case $\eta^{(i)}$ is a (unit) eigenvector of $A_2^*$ with eigenvalue $\lambda_i$.  Since the eigenvalues of $A_2^*$ are $\bar{\mu}$ and $\bar{\nu}$ it follows that $\|A_2^*\eta^{(i)}\|^2=|\mu|^2 ~{\rm or}~ |\nu|^2$.  Hence we can take $\beta=\left(\begin{smallmatrix}1\\0\end{smallmatrix}\right)$ so that $b_i(\beta)\geq|\sigma|^2>0$.


Now consider cases (iv) and (v).
\begin{center}
\begin{tabular}{ |c|c|c| }
\hline
&$A_1$ & $A_2$ \\ \hline \hline
(iv)&{$\begin{array}{c}
\\
\left(\begin{array}{cc}
1 &  0 \\
0 & d \\
\end{array}\right)\\
\\
\end{array} d\in \mathbb C, |d|\neq 1$} 
 & $\begin{array}{c} 
 \\
 \left(\begin{array}{cc}
1 &  |c| \\
c & 0 \\
\end{array}\right)c\in \mathbb C
\\
\\
\end{array}$ \\  \hline

(v)&{$\begin{array}{c}
\\
\left(\begin{array}{cc}
1 &  0 \\
0 & d \\
\end{array}\right)\\
\\
\end{array} d\in \mathbb C, |d|\neq 1$}  & $\begin{array}{c}\\
\left(\begin{array}{cc}
0 &  |c| \\
c & 0 
\end{array}\right)c\in \mathbb C
\\
\\
\end{array}$  \\ \hline

\end{tabular}
\end{center}

In cases (iv) and (v) we have, in Equation (\ref{eq3}),
\begin{eqnarray*}
a_i(\beta)&=&|\beta_1|^2 + |d|^2 |\beta_2|^2 -|\eta^{(i)}_1|^2 - |d|^2|\eta^{(i)}_2|^2\\
&=&(|\beta_1|^2-|\eta^{(i)}_1|^2) + |d|^2(|\beta_2|^2-|\eta^{(i)}_2|^2)\\
&=&(1-|d|^2)(|\eta^{(i)}_2|^2-|\beta_2|^2).
\end{eqnarray*}
If $|\eta^{(i)}_2| = 0 ~{\rm or}~1$ then $c=0$ and it reduces to the simultaneously diagonalizable case.  If $|\eta^{(i)}_2| \neq 0,1$ we can choose $\beta$ such that $|\beta_2| < |\eta^{(i)}_2|$ (resp. $|\beta_2| > |\eta^{(i)}_2|$) if $|d|<1$ (resp. $|d|>1$) to ensure that $a_i(\beta) > 0$ for $i=1,2$.


The methods used in cases (iv) and (v) can be adapted to the last case (vi):
\begin{center}
\begin{tabular}{ |c|c|c| }
\hline
&$A_1$ & $A_2$ \\ \hline \hline

(vi)&{$\begin{array}{c}
\\
\left(\begin{array}{cc}
d &  0 \\
0 & 1 \\
\end{array}\right)\\
\\
\end{array} d\in \mathbb C, |d|\neq 1$} 
 & $\begin{array}{c} 
 \\
 \left(\begin{array}{cc}
1 &  |c| \\
c & 0 \\
\end{array}\right)c\in \mathbb C
\\
\\
\end{array}$ \\  \hline
\end{tabular}
\end{center}

{\bf Step 2:} Show that, in each case, there is a choice of $(u,v)$ in the region characterized in Lemma \ref{18} for which the infimum in (\ref{6}) is, in fact, zero.

We choose $\hat{\beta}$ to ensure that $a_i(\hat{\beta})$ or $b_i(\hat{\beta})$ is positive as described in Step 1. 

Note that $g_{u,v}(\hat{\beta})$ vanishes at the two points $(u,v)=(\frac{1}{\|A_1^*(\hat{\beta})\|},0)$, $(u,v)=(0,\frac{1}{\|A_2^*(\hat{\beta})\|})$ and also along a curve joining these two points.

We now consider two cases:

{\bf Case (i):} $a_i(\hat{\beta})>0$

Choose $(u_0,v_0)$ such that $0< v_0 <\frac{1}{\|A_2^*\|}$, $f_i(u_0,v_0,\hat{\beta})>0$ and $g_{u_0,v_0}(\hat{\beta})=0$.  This is possible using Lemma \ref{18} and the above note about the vanishing of $g_{u,v}(\hat{\beta})$.

Let 
$$x_0 = \inf\{u:\inf_\beta g_{u,\lambda_0 u}(\beta)\leq 0\}~~{\rm where}~~\lambda_0=\frac{v_0}{u_0}.$$
Note that $x_0^2\geq \frac{1}{\|A_1^*\|^2 + \lambda_0^2 \|A_2^*\|^2}$. Also, from Lemma \ref{18}, it is clear that $f_i(x_0,\lambda_0 x_0,\hat{\beta}) >0$.

We now show that $\inf_\beta g_{x_0,\lambda_0 x_0}(\beta)=0$.  

To prove this we first show that $g_{(x_0,\,\lambda_0 x_0)}(\beta)\geq 0$ for all
$\beta$ (with $\|\beta\|_{2}=1$) as follows.
Assume there exists $\beta=\mu$ such that $g_{(x_0,\lambda_0
x_0)}(\mu)<0.$ Then there exists a neighborhood $U$ of
$x_0$ such that $g_{u,\lambda_0 u}(\mu)<0$ for all $u
\in U.$ For any $u\in U,$ $\inf_{\beta} g_{u,\lambda_0
u}(\beta)< 0,$ since $g_{u,\lambda_0 u}(\mu)<0$  for all $u\in U.$  Since $U$ is a neighborhood of
$x_0$ there exists a $u\in U$ such that $u<x_0.$ By the
previous assertion, $\inf_\beta g_{u,\lambda_0 u}(\beta)\leq 0$ for this smaller value of $u$,  which is a contradiction. 

Since $\inf_\beta g_{x_0,\lambda_0 x_0}(\beta)\leq 0$ by the definition of $x_0$ it follows that $\inf_{\beta} g_{x_0,\lambda_0 x_0}(\beta)=0.$

{\bf Case (ii):} $b_i(\hat{\beta})>0$

The arguments in this case are similar to Case (i).  This time choose $(u_0,v_0)$ such that $0< u_0 <\frac{1}{\|A_1^*\|}$, $f_i(u_0,v_0,\hat{\beta})>0$ and $g_{u_0,v_0}(\hat{\beta})=0$.  

Let 
$$y_0 = \inf\{v:\inf_\beta g_{\lambda_0 v,v}(\beta)\leq 0\}~~{\rm where}~~\lambda_0=\frac{u_0}{v_0}.$$
As in Case (i) we can see that $y_0^2\geq \frac{1}{\lambda_0^2\|A_1^*\|^2 + \\|A_2^*\|^2}$ and (from Lemma \ref{18}) that $f_i(\lambda_0 y_0,y_0,\hat{\beta}) >0$.

Using a procedure similar to that used in Case (i) it follows that $\inf_\beta g_{\lambda_0 y_0,y_0}(\beta)=0$.  

We have therefore shown that for all the cases in Table 3 which were not covered by the operator space approach it is possible to choose $(u,v)$ such that the infimum in (\ref{6}) is zero and this infimum is attained at a vector $\beta$ not equal to $\eta^{(1)}$ or $\eta^{(2)}$, so that the last term in parenthesis in (\ref{6}) is positive at $\beta$.  

It follows that, in each of these cases, there exists a contractive homomorphism which is not completely contractive.



\section{An Interesting Operator Space Computation}

In Section \ref{opspace} the existence of contractive homomorphisms which are not completely contractive was shown in many cases by studying different isometric embeddings of the space $(\mathbb C^2, \|\cdot\|_\A)$ into $(\mathcal M_2, \|\cdot \|_{\rm op})$ which led to distinct operator space structures.  The two embeddings considered there were $(z_1,z_2)\mapsto z_1 A_1 + z_2 A_2$ and $(z_1,z_2)\mapsto z_1 A_1^{\rm t} + z_2 A_2^{\rm t}$.  In this section we show that we can, for some choices of $(A_1,A_2)$, construct large collections of isometric embeddings of the space $(\mathbb C^2, \|\cdot\|_\A)$ into various matrix spaces.  Although the embeddings are into very distinct matrix spaces, we show that the operator space structures thus obtained are equivalent.

A result which is very useful in this context is the following proposition due to Douglas, Muhly and Pearcy (cf. \cite [Prop. 2.2]{DMP}).

\begin{prop}\label{DMP}
For $i=1,2,$ let $T_i$ be a contraction on a Hilbert space $\mathcal{H}_i$ and let $X$ be an operator mapping $\mathcal{H}_2$ into $\mathcal{H}_1$.  A necessary and sufficient condition that the operator on $\mathcal{H}_1 \oplus \mathcal{H}_2$ defined by the matrix $\left(\begin{smallmatrix}T_1 & X\\
0 & T_2\end{smallmatrix}\right)$ be a contraction is that there exist a contraction $C$ mapping $\mathcal{H}_2$ into $\mathcal{H}_1$ such that 
$$X=\sqrt{1_{\mathcal{H}_1} - T_1 T_1^*}~C~\sqrt{1_{\mathcal{H}_2} - T_2^* T_2}.$$
\end{prop}


 The operator norm of the block matrix $\left(\begin{smallmatrix}\alpha I_m & B\\
0 & \alpha I_n\end{smallmatrix}\right),$ where $B$ is an $m \times
n$ matrix and $\alpha \in \mathbb C$, is not hard to compute (cf.
\cite [Lemma 2.1]{G}). The result can be easily extended to a matrix of the form
$\left(\begin{smallmatrix}\alpha_1 I_m & B\\
0 & \alpha_2 I_n\end{smallmatrix}\right),$ for arbitrary $\alpha_1,\alpha_2 \in \mathbb C$. 
\begin{lem}\label{lem:con}
If $B$ is an $m\times n$ matrix and $\alpha_1, \alpha_2 \in \mathbb C$ then $$\Big\|\left(\begin{matrix}\alpha_1 I_m & B\\
0 & \alpha_2 I_n\end{matrix}\right)\Big\| =\Big\|\left(\begin{matrix}\alpha_1  & \|B\|\\
0 & \alpha_2\end{matrix}\right)\Big\|.$$
\end{lem}
\begin{proof}
Consider the following two sets
$$S_1 = \big\{\big((\alpha_1,\alpha_2);B\big):\big\|\left(\begin{smallmatrix}\alpha_1 I_m & B\\
0 & \alpha_2 I_n \end{smallmatrix}\right)\big\| \leq 1 \big\}$$
and 
$$S_2 = \big\{\big((\alpha_1,\alpha_2);B\big):\big\|\left(\begin{smallmatrix}\alpha_1 & \|B\|\\
0 & \alpha_2 \end{smallmatrix}\right)\big\| \leq 1 \big\}.$$
To prove the lemma, it is sufficient to show that these unit balls are the same.  

From Proposition \ref{DMP} the condition for the contractivity of the elements of $S_1$ and $S_2$ is the same, that is,
$$\|B\|^2 \leq (1-|\alpha_1|^2)(1-|\alpha_2|^2)$$

\end{proof}

The important observation from the lemma above is that, for fixed $\alpha_1, \alpha_2$, the norm of the matrix $\left(\begin{smallmatrix}\alpha_1 I_m & B\\
0 & \alpha_2 I_n\end{smallmatrix}\right)$ depends only on $\|B\|$.  

Now consider the pair $\mathbf A = (A_1,A_2)$ with $A_1 = \left( \begin{smallmatrix} \alpha_1 & 0\\
0 & \alpha_2 \end{smallmatrix}\right), A_2 = \left( \begin{smallmatrix} 0 & \beta\\
0 & 0 \end{smallmatrix}\right)$. Given any $m\times n$ matrix $B$ with $\|B\|= |\beta|$ we have the following isometric embedding of $(\mathbb C^2, \|\cdot \|_\A)$ into $(\mathcal M_{m+n}, \|\cdot \|_{\rm op})$
$$(z_1,z_2)\mapsto \left( \begin{matrix}
z_1 \alpha_1 I_m & z_2 B\\
0 & z_1 \alpha_2 I_n
\end{matrix}\right).$$
For various choices of the dimensions $m,n$ and the matrix $B$, this represents a large collection of isometric embeddings.

For fixed $\alpha_1, \alpha_2$, we let $X_B$ represent the above embedding of $(\mathbb C^2, \|\cdot\|_\A)$ into $(\mathcal{M}_{m+n}, \|\cdot\|_{\rm op})$. We now show that the operator space structures determined by these embeddings depend only on $\|B\|$.  If $\mathcal{V}_\A$ is the space $(\mathbb C^2, \|\cdot\|_\A)$, then $(X_B \otimes I_k)$ gives the embedding of $\mathcal {M}_k(\mathcal{V}_\A)$ into $\mathcal {M}_k(\mathcal {M}_{m+n}(\mathbb {C}))$.  An element of $\mathcal {M}_k(\mathcal{V}_\A)$ is defined by a pair of $k\times k$ matrices $Z_1, Z_2$, and the corresponding embedding into $\mathcal {M}_k(\mathcal {M}_{m+n}(\mathbb {C}))$ has the form 
$$\left(\begin{matrix}\alpha_1Z_1 \otimes I_m & Z_2\otimes B\\
0 & \alpha_2Z_1 \otimes I_n \end{matrix}\right). $$

It now remains to show that the operator norm of this matrix depends only on $\|B\|$.  Using Proposition \ref{DMP} it can be shown that 
$$\Bigg\| \left(\begin{matrix}\alpha_1Z_1 \otimes I_m & Z_2\otimes B\\
0 & \alpha_2Z_1 \otimes I_n \end{matrix}\right)\Bigg\| \leq 1 ~{\rm if~and~only~if}~ \Bigg\|\left(\begin{matrix}\alpha_1Z_1 & Z_2 \|B\|\\
0 & \alpha_2Z_1 \end{matrix}\right)\Bigg\| \leq 1.$$
Hence it follows that these two norms are in fact equal. We have therefore proved the following theorem. 
\begin{thm} For all $m\times n$ matrices $B$ with the same (operator) norm, the operator space structures on $\mathbb C^{m+n},$ determined by the different embeddings $$(z_1,z_2) \mapsto  
z_1 \Big ( \begin{array}{cc} \alpha_1 I_m & 0\\ 0 & \alpha_2 I_n \end{array} \Big ) + z_2 \Big ( \begin{array}{cc} 0 &  B\\ 0 & 0  \end{array} \Big ), \,\,  \alpha_1, \alpha_2\in \mathbb C,$$ 
are completely isometric irrespective of the particular choice of $B.$ Moreover all of them are completely isometric to the MIN space.
\end{thm}

\subsection*{Acknowledgment} The authors gratefully acknowledge the help they have received from Sayan Bagchi, Michael Dritschel and Dmitry Yakubovich.

\end{document}